\documentclass[reqno]{amsart}[11pt,draft,amssymb]
\usepackage[all]{xy}
\usepackage{amsthm,array,amssymb,amscd,amsfonts,latexsym}
\xyoption{poly}

{\catcode`\@=11
\gdef\n@te#1#2{\leavevmode\vadjust{%
 {\setbox\z@\hbox to\z@{\strut#1}%
  \setbox\z@\hbox{\raise\dp\strutbox\box\z@}\ht\z@=\z@\dp\z@=\z@%
  #2\box\z@}}}
\gdef\leftnote#1{\n@te{\hss#1\quad}{}}
\gdef\rightnote#1{\n@te{\quad\kern-\leftskip#1\hss}{\moveright\hsize}}
\gdef\?{\FN@\qumark}
\gdef\qumark{\ifx\next"\DN@"##1"{\leftnote{\rm##1}}\else
 \DN@{\leftnote{\rm??}}\fi{\rm??}\next@}}
\DeclareFontFamily{OT1}{wncyr}{\hyphenchar\font45
}
\DeclareFontShape{OT1}{wncyr}{m}{n}{%
   <5> <6> <7> <8> <9> gen * wncyr
   <10> <10.95> <12> <14.4> <17.28> <20.74>  <24.88>wncyr10}{}
\DeclareFontShape{OT1}{wncyr}{m}{it}{%
   <5> <6> <7> <8> <9> gen * wncyi
   <10> <10.95> <12> <14.4> <17.28> <20.74> <24.88> wncyi10}{}
\DeclareFontShape{OT1}{wncyr}{m}{sc}{%
   <5> <6> <7> <8> <9> <10> <10.95> <12> <14.4>
   <17.28> <20.74> <24.88>wncysc10}{}
\DeclareFontShape{OT1}{wncyr}{b}{n}{%
   <5> <6> <7> <8> <9> gen * wncyb
   <10> <10.95> <12> <14.4> <17.28> <20.74> <24.88>wncyb10}{}
\input cyracc.def
\def\rus{\usefont{OT1}{wncyr}{m}{n}\cyracc\fontsize{8}{10pt}\selectfont}

\DeclareMathSizes{9}{9}{7}{5} 

\theoremstyle{plain}

\newtheorem*{thmnonumber}{Theorem}

\theoremstyle{definition}

\newtheorem{nothing*}[theorem]{}
\newtheorem{subnothing*}[sub]{}

\newtheorem*{remarknonumber}{Remark}
\newtheorem*{examplenonumber}{Example}

\theoremstyle{remark}

\begin{document}




\title[Algebraic cones]
{Algebraic cones}
\author[Vladimir  L. Popov]{Vladimir  L. Popov${}^*$}
\address{Steklov Mathematical Institute,
Russian Academy of Sciences, Gubkina 8,
Moscow\\ 119991, Russia}
\email{popovvl@orc.ru}

\thanks{
 ${}^*$\,Supported by
 grants {\rus RFFI
08--01--00095}, {\rus N{SH}--1987.2008.1}, and the
program {\it Contemporary Problems of Theoretical
Mathematics} of the
Russian Academy of Sciences, Branch of Mathematics.
}

\date{August 7, 2009}



\maketitle

\begin{abstract}
A characterization of algebraic cones in terms of actions of the one-dimensional multiplicative algebraic monoid ${\bf M}_{m}$ and the algebraic group ${\bf G}_{\rm m}$
are given.
\end{abstract}

\vskip 10mm

This note answers a question of K. Adjamagbo asked in \cite{A}.

Below all algebraic varieties are taken over an algebraically closed field $k$. 

An irreducible algebraic variety $X$ 
is called {\it a cone} 
if $X$ is affine and its 
coordinate algebra 
$k[X]$ admits a connected ${\bf N}$-grading: 
\begin{equation}\label{grad}
k[X]=\bigoplus_{d\in {\bf N}} k[X]_d,\qquad k[X]_0=k.
\end{equation} 

Let ${\bf M}_{\rm m}$ be the multiplicative algebraic monoid  whose underlying variety is 
the affine line ${\bf A}^1$ and the multiplication $\mu\colon {\bf A}^1\times {\bf A}^1\to {\bf A}^1$ in defined by that in $k$, i.e., $\mu^*(T)=T\otimes T$ where $T$ is the standard coordinate function on ${\bf A}^1$. The group of 
units of ${\bf M}_{\rm m}$ is ${\bf M}_{\rm m}\setminus \{0\}={\bf G}_{\rm m}$. We put 
\begin{equation}\label{chii}
\chi^{}_d\colon {\bf G}_{\rm m}\longrightarrow {\bf G}_{\rm m},\qquad t\mapsto t^d,
\end{equation} 

One says that ${\bf M}_{\rm m}$ {\it acts} on a variety $Y$ if a morphism $\alpha\colon {\bf M}_{\rm m}\times Y\to Y$ is given, such that $\alpha(g, \alpha(h, y))=\alpha (gh, y)$ and $\alpha(1, y)=y$ for all $g, h, \in {\bf M}_{\rm m}$, $y\in Y$. We write 
$g(y):=\alpha(g, y)$. The restriction of $\alpha$ to ${\bf G}_{\rm m}\times Y$ is the usual group action of ${\bf G}_{\rm m}$ on $Y$. The set 
${\bf M}_{\rm m}(y):=\{g(y)\mid g\in {\bf M}_{\rm m}\}$ is called 
an ${\bf M}_{m}$-{\it orbit} of $y$
(warning: different ${\bf M}_{m}$-orbits may have a nonempty intersection). If ${\bf M}_{\rm m}(y)=y$, then $y$ is called a {\it fixed point} of the action.

\begin{thmnonumber} 
Let $X$ be an irreducible algebraic variety. 
Consider the properties:
\begin{enumerate}
\item[\rm(i)] 
$X$ is a cone;
\item[\rm(ii)] there is an action of $\;{\bf M}_{\rm m}$ on 
$X$ with a unique fixed point;
\item[\rm(iii)] there is an action of $\;{\bf G}_{\rm m}$ on 
$X$ with a 
fixed point 
and 
without other 
closed orbits.
\end{enumerate}
Then {\rm (i)}\!$\Rightarrow$\!{\rm (ii)}\!$\Rightarrow$\!{\rm (iii)} and, 
if $X$ is normal, {\rm (iii)}\!$\Rightarrow$\!{\rm (i)}.
\end{thmnonumber}
\begin{proof} We may assume that  $\dim X>0$.


(i)\!$\Rightarrow$\!(ii) Let $X$ be a cone. Consider a grading \eqref{grad}. 
Then formula $t\cdot f:=t^df$ for $t\in {\bf G}_m$, $f\in k[X]_d$, defines an action of
${\bf G}_m$ on $k[X]$ by $k$-algebra automorphisms. In turn, it defines an (algebraic) action of ${\bf G}_m$ on $X$. 
As grading \eqref{grad} is connected, 
\begin{equation}\label{fixedd}
k[X]^{{\bf G}_{\rm m}}=k.
\end{equation} 

Since $k[X]^{{\bf G}_{\rm m}}$ separates closed ${\bf G}_{\rm m}$-orbits (see \cite[Cor.\;A1.3]{MF}),
from \eqref{fixedd} and \cite[Cor.\;I.1.8]{Bor} we deduce that there is a unique such orbit $O$. 
The coordinate algebra $k[O]$ of $O$ does not contain ${\bf G}_{\rm m}$-stable proper ideals
as the zero set of such an ideal would be 
a proper ${\bf G}_{\rm m}$-stable subset of $O$. Hence, as closedness of $O$ implies surjectivity of the map $k[X]\to k[O]$, $f\mapsto f|_O$, the ${\bf G}_{\rm m}$-stable ideal $\bigoplus_{d>0} k[X]_d$ vanishes on $O$.
Therefore, $O$ is a single point.

By the embedding theorem \cite[Theorem 1.5]{popov-vinberg}  we may assume that $X$ is a 
closed ${\bf G}_{\rm m}$-stable subset of a finite dimensional ${\bf G}_{\rm m}$-module $V$ and $X$ is not contained in its proper 
submodule.
But the set of zeros of all linear functions on $V$ that vanish on $X$ is a ${\bf G}_{\rm m}$-submodule containing $X$. Hence 
we have the embedding 
 $V^*\hookrightarrow k[X]$, $\ell\mapsto \ell|_X$,
of 
${\bf G}_{\rm m}$-modules. As the weights of $V$ are the inverses of that of $V^*$,  
by \eqref{fixedd} this yields that the origin $0_V$ is the unique ${\bf G}_{\rm m}$-fixed point in $V$; in particular, $O=0_V$.
By \eqref{grad} this also yields that,  for every ${\bf G}_{\rm m}$-weight $\chi^{}_d$ of $V^*$, we have $d\in {\bf N}$. Hence by \eqref{chii} there are
$d_1,\ldots, d_n\in {\bf N}$ and a basis $e_1,\ldots, e_n$ of $V$ such that 
$t(e_i)=t^{d_i}e_i$ for every $t\in {\bf G}_{\rm m}$ and $i$. This shows that the action of ${\bf G}_m$ on $V$ extends to the action of 
${\bf M}_{\rm m}$ on $V$ by putting 
$0(v)=0_V$ for every $v\in V$. As  ${\bf M}_{\rm m}$ is the closure of ${\bf G}_{\rm m}$
and $X$ is ${\bf G}_{\rm m}$-stable and closed, $X$ is ${\bf M}_{\rm m}$-stable as well. By the same reason, $0_V$ is a fixed point of ${\bf M}_{\rm m}$. There are no other such points because $V^{{\bf G}_{\rm m}}=\{0_V\}$.  

\vskip 1.5mm

(ii)\!$\Rightarrow$\!(iii) Assume that 
${\bf M}_m$ acts on an irreducible 
variety $X$ with a unique fixed point $x$. 
Restrict this action to ${\bf G}_m$.
Take a point $y\in X$.
If the ${\bf G}_m$-orbit ${\bf G}_m(y)$ is closed in $X$, then ${\bf G}_m(y)={\bf M}_m(y)$  since ${\bf G}_m$ is dense in ${\bf M}_m$.

We claim that 
then ${\bf G}_m(y)$ is a single point, whence
${\bf G}_m(y)=x$ because of the uniqueness of $x$.
Indeed, assume the contrary, i.e.,
$\dim {\bf G}_m(y)=1$.  The orbit map
$\varphi^{}_y\colon {\bf M}_{\rm m}\to {\bf M}_m(y)$,
\  $g\mapsto g(y)$,
 is 
a surjective morphism of one-dimensional smooth  varieties
 such that every fiber is finite 
and, for every point $z\in {\bf M}_{\rm m}(y)$, 
$z\neq \varphi^{}_y(0)$, the cardinality of 
$\varphi^{-1}_y(z)$ is equal to 
the order $s$ of the ${\bf G}_{\rm m}$-stabilizer
of $y$ while the cardinality of $\varphi^{-1}_y(\varphi^{}_y(0))$ is $s+1$. By \cite[Sect.\;2, Cor.\;2]{G}
this is impossible, a contradiction.
Thus, $x$ is a unique closed ${\bf G}_m$-orbit.

\vskip 1.5mm


(iii)\!$\Rightarrow$\!(i) Now assume that  ${\bf G}_{\rm m}$ acts on $X$ with a unique fixed point $x$ and without other closed orbits. Then, for every point $y\in X$, $y\neq x$,
the closure $\overline{{\bf G}_m(y)}$ of the orbit ${\bf G}_m(y)$ is one-dimensional and 
\begin{equation}\label{x}
x\in\overline{{\bf G}_m(y)}. 
\end{equation}

Assume further that $X$ is normal. Then by Sumihiro's theorem \cite[Cor.\;2 of Lemma\;8]{S} there is a ${\bf G}_{\rm m}$-stable affine open neighborhood $U$ of $x$. We claim that $X=U$. Indeed, if not, then $X\setminus U$ is a nonempty ${\bf G}_{\rm m}$-stable closed subset in $X$ and \eqref{x} is impossible for $y\in X\setminus U$, a contradiction. 


Thus, $X=U$, hence $X$ is affine.

As elements of 
$k[X]^{{\bf G}_{\rm m}}$ are constant on ${\bf G}_{\rm m}$-orbits, 
\eqref{x} implies that $f(y)=f(x)$ for every $f\in k[X]^{{\bf G}_{\rm m}}$ and $y\in X$;
whence \eqref{fixedd} holds.

Now let $k[X]_{d}$ be the $\chi^{}_d$-isotypic component 
of the ${\bf G}_{\rm m}$-module $k[V]$. 
Then 
\begin{equation}\label{Z}
k[X]=\bigoplus_{d\in{\bf Z}}k[X]_d
\end{equation}
is a ${\bf Z}$-grading of the $k$-algebra 
$k[X]$. By \eqref{fixedd} it is connected, i.e., 
$k[X]_0=k$. 

We claim that there are no integers $d_1>0$ and $d_2<0$ such that $k[X]_{d_i}\neq 0$ for $i=1, 2$. Indeed, assume the contrary. Then there is a point $y\in X$ such that $k[X]_{d_i}$, $i=1, 2$, is not in the kernel of $k[X]\to k[\overline{{\bf G}_m(y)}]$, $f\mapsto f|_{\overline{{\bf G}_m(y)}}$. Hence, for every $i=1, 2$, the $\chi_{d_i}$-isotypic component of 
the ${\bf G}_m$-module $k[\overline{{\bf G}_m(y)}]$ 
is nonzero. This implies that there is an integer $d\neq 0$ such that 
the ${\bf G}_m$-stable maximal ideal $\{f\in k[\overline{{\bf G}_m(y)}]\mid f(x)=0\}$ of 
$k[\overline{{\bf G}_m(y)}]$ 
has the nonzero $\chi^{}_d$- and $\chi^{-1}_d$-isotypic components.
Let 
$p$ and $q$ be the nonzero elements of resp. the first and second of them. Then 
$pq$ is constant on ${\bf G}_m(y)$ and hence on 
$\overline{{\bf G}_m(y)}$. As $pq(x)=0$, this means that $pq=0$ contrary to the irreducibility of
$\overline{{\bf G}_m(y)}$, a contradiction.

Thus,  $k[X]_d=0$ in \eqref{Z} either for all negative or for all positive $d$'s. Repla\-cing, if necessary, the action of ${\bf G}_m$ on $X$ by  $g\cdot y:=g^{-1}(y)$, we may assume that the first possibility is realized, i.e., \eqref{grad} holds. This completes the proof.  
\quad $\square$
\renewcommand{\qed}{}\end{proof}

\begin{remarknonumber} The following example shows that, in general, without the assumption of normality of $X$ the implications (iii)\!$\Rightarrow$\!(i) and 
(ii)\!$\Rightarrow$\!(i)
do not hold.
\begin{examplenonumber}
Let $X$ be the image of 
the morphism
\begin{gather*}
\nu\colon {\bf P}^1\longrightarrow {\bf P}^2,\hskip 2mm (a_0:a_1)\mapsto (p^3:q^2t-p^3:q^3-qp^2),
\quad p=a_1-a_0,\; q=a_1+a_0.
\end{gather*}
$X$ is the projective plane cubic with an ordinary double point $O=(1:0:0)$ and ${\bf P}^1\to X$, $x\mapsto \nu(x)$, is the normalization map. Formula $\alpha(t, (a_0:a_1))=(a_0:ta_1)$ defines an action of ${\bf G}_{\rm m}$ on ${\bf P}^1$ that descends to $X$ 
by means of $\nu$, see \cite{P}. For this action, $O$ is a unique fixed point and $X\setminus O$ is a ${\bf G}_{\rm m}$-orbit. This action extends to the one of ${\bf M}_{\rm m}$ by putting $0(x)=O$ for every $x\in X$.
\end{examplenonumber} 
\end{remarknonumber}

\end{document}